\documentclass{patmorin}
\usepackage{pat}
\usepackage[T1]{fontenc}
\usepackage[utf8]{inputenc}
\usepackage{todonotes}

\usepackage[mathlines]{lineno}
\newcommand*\patchAmsMathEnvironmentForLineno[1]{%
 \expandafter\let\csname old#1\expandafter\endcsname\csname #1\endcsname
 \expandafter\let\csname oldend#1\expandafter\endcsname\csname end#1\endcsname
 \renewenvironment{#1}%
    {\linenomath\csname old#1\endcsname}%
    {\csname oldend#1\endcsname\endlinenomath}}%
\newcommand*\patchBothAmsMathEnvironmentsForLineno[1]{%
 \patchAmsMathEnvironmentForLineno{#1}%
 \patchAmsMathEnvironmentForLineno{#1*}}%
\AtBeginDocument{%
\patchBothAmsMathEnvironmentsForLineno{equation}%
\patchBothAmsMathEnvironmentsForLineno{align}%
\patchBothAmsMathEnvironmentsForLineno{flalign}%
\patchBothAmsMathEnvironmentsForLineno{alignat}%
\patchBothAmsMathEnvironmentsForLineno{gather}%
\patchBothAmsMathEnvironmentsForLineno{multline}%
}

\usepackage[longnamesfirst,numbers,sort&compress]{natbib}

\definecolor{brightmaroon}{rgb}{0.76, 0.13, 0.28}
\definecolor{linkblue}{rgb}{0, 0.337, 0.227}
\newcommand{\defin}[1]{\emph{\textcolor{brightmaroon}{#1}}}
\makeatletter
\def\mathcolor#1#{\@mathcolor{#1}}
\def\@mathcolor#1#2#3{%
  \protect\leavevmode
  \begingroup
    \color#1{#2}#3%
  \endgroup
}
\makeatother
\newcommand{\mathdefin}[1]{\mathcolor{brightmaroon}{#1}}
\DeclareMathOperator{\pw}{pw}

\DeclareMathOperator{\ply}{ply}
\DeclareMathOperator{\bn}{bn}

\newcommand{\CC}{\mathcal{C}}
\newcommand{\sop}[2]{#1\langle#2\rangle}
\newcommand{\ssop}[2]{#1\lbrace#2\rbrace}
\newcommand{\sm}{\smallsetminus}

\usepackage[inline]{enumitem}

\title{\MakeUppercase{Basis Number and Pathwidth}}

\author{
 Babak Miraftab%
  \thanks{School of Computer Science, Carleton University.},\,
 Pat Morin\footnotemark[1] ,\, and
 Yelena Yuditsky%
 \thanks{Computer Science Department, Université libre de Bruxelles and School of Computer Science, University of Leeds.}
}

\begin{document}
\maketitle

\begin{abstract}
  We prove two results relating the basis number of a graph $G$ to path decompositions of $G$.  Our first result shows that the basis number of a graph is at most four times its pathwidth. Our second result shows that, if a graph $G$ has a path decomposition with adhesions of size at most $k$ in which the graph induced by each bag has basis number at most $b$, then $G$ has basis number at most $b+O(k\log^2 k)$.  The first result, combined with recent work of {\NoHyper\citeauthor{geniet.giocanti:basis}} shows that the basis number of a graph is bounded by a polynomial function of its treewidth. The second result (also combined with the work of {\NoHyper\citeauthor{geniet.giocanti:basis}}) shows that every $K_t$-minor-free graph has a basis number bounded by a polynomial function of $t$.
\end{abstract}

\section{Introduction}
For two graphs\footnote{Any graph $G$ that we consider is finite, simple, and undirected with vertex set $V(G)$ and edge set $E(G)$.} $G_1$ and $G_2$, the \defin{symmetric difference} $G_1\oplus G_2$ denotes the graph with vertex set $V(G_1\oplus G_2)=V(G_1)\cup V(G_2)$ that contains each edge $e\in E(G_1)\cup E(G_2)$ if and only if $e$ appears in exactly one of $G_1$ or $G_2$.
Let $G$ be a graph and let $\mathcal{B}$ and $\mathcal{C}$ be sets of subgraphs of $G$. We say that $\mathcal{B}$ \defin{generates} $\mathcal{C}$ if, for each $C\in\mathcal{C}$, there exists $\mathcal{B}'\subseteq\mathcal{B}$ such that $C=\bigoplus_{B\in\mathcal{B}'} B$.
In this case, we say that $\mathcal B$ is a \defin{generating set} of $\CC$.
The \defin{cycle space} of $G$ denoted by $\mathcal C(G)$ is the set generated by all cycles of $G$.
A graph $H$ is \defin{Eulerian} if each vertex of $H$ has even degree. Veblen's Theorem \cite{veblen:application} states that the edges of an Eulerian graph $H$ can be covered with pairwise edge-disjoint cycles. From this, the cycle space $\CC(G)$ of $G$ consists of all Eulerian subgraphs of $G$.
A \defin{basis} of the cycle space consists of a minimum set of Eulerian subgraphs from which we can generate $\CC(G)$.

Let $G$ be a graph and let $\mathcal{B}$ be a generating set for $\CC(G)$. 
For an edge $e$ of $G$, the \defin{ply}\footnote{Various works have different names for what we call \emph{ply}, including \defin{charge}, \defin{congestion}, \defin{participation}, and \defin{load}.  What we call a $k$-basis is also referred to as a \defin{$k$-fold cycle basis}.} of $e$ in $\mathcal{B}$, $\mathdefin{\ply(e,\mathcal{B})}:=|\{B\in\mathcal{B}:e\in E(B)\}|$ is the number of elements in $\mathcal{B}$ that include $e$.  
The \defin{ply} of $\mathcal{B}$, $\mathdefin{\ply(\mathcal{B})}:=\max\{\ply(e,\mathcal{B}):e\in E(G)\}$, is the maximum ply in $\mathcal{B}$ of any edge of $G$.  
Let $\mathcal{B}$ be a generating set (or basis) for $\CC(G)$.
Then we say that $\mathcal B$ is a \defin{$k$-generating set (k-basis)} if $\ply(e,\mathcal{B})\le k$ for every $e\in E(G)$. 
The \defin{basis number} of $G$, $\mathdefin{\bn(G)}:=\min\{\ply(\mathcal{B}):\text{$\mathcal{B}$ is a cycle basis for $G$}\}$, is the minimum ply of any cycle basis of $G$.

It is a well-known fact in linear algebra that every generating set of a vector space contains a basis.  For this reason, in this paper, we focus on $k$-generating sets and not $k$-bases.

The basis number of a graph is a measure of its complexity.  A graph $G$ is a forest\footnote{A \defin{forest} is a graph with no cycles.} if and only if $\bn(G)=0$. A graph $G$ is a cactus graph\footnote{A \defin{cactus graph} is a graph whose cycles are pairwise edge-disjoint.} if and only if $\bn(G)\le 1$.  The most famous result of this type is \defin{Mac~Lane's Planarity Criterion} \cite{MacLane1937} which states that a graph $G$ is planar if and only if $\bn(G)\le 2$. 

Results for graphs with basis number larger than $2$ are generally less sharp than in the planar case. 
In this direction, \citet{lehner.miraftab:sparse} investigate graphs that admit embeddings on surfaces. 
They prove that every non-planar graph embeddable on a surface of Euler characteristic $0$ (that is, the torus or the Klein bottle) has basis number $3$. 
More generally, if a graph $G$ embeds on a surface of genus $g\ge 2$, then $\bn(G)\in O(\log^2 g)$ \cite{lehner.miraftab:sparse}. 
Complete graphs satisfy $\bn(K_n)\le 3$ (see \cite{MR615307}), whereas $1$-planar graphs\footnote{A graph is \defin{$1$-planar} if it can be embedded in $\R^2$ so that each edge is crossed at most once.} (see \cite{biedl}) have unbounded basis number.
Motivated by applications to error correcting codes in quantum computing, \citet{freedman.hastings:building} establish the following unconditional result:

\begin{thm}[\citet{freedman.hastings:building}]\label{general_upper_bound}
  For every $n$-vertex graph $G$, $\bn(G)\in O(\log^2 n)$.
\end{thm}

In the current paper, we relate the basis number of a graph to its pathwidth, a graph parameter that is essentially independent of embeddability on surfaces.\footnote{There are graphs of pathwidth $4$ (containing many pairwise vertex-disjoint copies of $K_5$) that cannot be embedded on any surface of bounded genus.  There are planar graphs (grids) that have unbounded pathwidth.}  A \defin{path decomposition} of a graph $G$ is a sequence $B_0,\ldots,B_n$ of subsets of $V(G)$ such that
\begin{enumerate*}[label=(\roman*)]
  \item for each vertex $v$ of $G$, and each $0\le i < k\le n$, $v\in B_i\cap B_k$ implies that $v\in B_j$ for each $i\le j\le k$; and
  \item for each edge $vw$ of $G$, there exists some $i\in\{0,\ldots,n\}$ such that $\{v,w\}\subseteq B_i$.
\end{enumerate*}
The \defin{width} of a path decomposition $B_0,\ldots,B_n$ is $\max\{|B_i|:i\in\{0,\ldots,n\}\}-1$.
The \defin{pathwidth}, \defin{$\pw(G)$} of a graph $G$ is the minimum width of a path decomposition of $G$.  Our main result is the following relationship between pathwidth and basis number, which we prove in \cref{proof_of_one}:

\begin{thm}\label{main}
  For any graph $G$, $\bn(G)\le 4\pw(G)$.
\end{thm}

In parallel with the current work, \citet{geniet.giocanti:basis} study the basis number of $K_t$-minor-free graphs and establish that there exist a function $f:\N\to\N$ such that every $K_t$-minor-free graph $G$ has $\bn(G)\le f(t)$. The function $f$ is unspecified, but is at least doubly-exponential.

The Graph Minor Structure Theorem \cite{robertson.seymour:gmxvi} characterizes $K_t$-minor-free graphs in terms of tree decompositions.\footnote{A \defin{tree decomposition}  of a graph $G$ is a collection $(B_x:x\in V(T))$ of subsets of $V(G)$ indexed by the elements of a tree $T$ such that \begin{enumerate*}[label=(\roman*)]
  \item for each vertex $v$ of $G$, $T[x\in V(T):v\in B_x]$ is non-empty connected subtree of $T$; and
  \item for each edge $vw$ of $G$, there exists some $x\in V(T)$ such that $\{v,w\}\subseteq B_x$.
\end{enumerate*}
The \defin{width} of a tree decomposition $(B_x:x\in V(T))$ is $\max\{|B_x|:x\in V(T)\}-1$. The \defin{adhesions} of a tree decomposition $(B_x:x\in V(T))$ are the sets in $\{B_x\cap B_y:xy\in E(T)$.} A natural starting point, therefore, is to study the basis number of graphs having bounded treewidth.  \citet{geniet.giocanti:basis} do exactly this and show that every graph $G$ of treewidth at most $k$ has $\bn(G)\le k^{2^{O(k^2)}}$. Their proof reduces the bounded treewidth case to the bounded pathwidth case and then establishes a version of \cref{main} with a much larger dependence on $\pw(G)$.  
Using \Cref{main} instead establishes that every graph $G$ of treewidth $k$ has $\bn(G)\in O(k^5)$.

The tree decompositions that appear in the Graph Minor Structure Theorem do not have bounded width, but they do have adhesions of bounded size.  
Thus, {\NoHyper\citeauthor{geniet.giocanti:basis}}'s result on $K_t$-minor-free graphs relies on a generalization of \cref{main} to path decompositions with adhesions of bounded size. This motivates our second result, which we prove in \cref{proof_of_two}.\footnote{For a graph $G$ and a subset $X$ of vertices of $G$, the notation \defin{$G[X]$} denotes the induced subgraph of $G$ with vertex set $V(G[X]):=X$ and edge set $E(G[X]):=\{vw\in E(G):\{v,w\}\subseteq X\}$.  For a graph $G$ and a set $S$ of edges of $G$, the notation \defin{$G-S$} refers to the graph with vertex set $V(G-S):=V(G)$ and edge set $E(G-S):=E(G)\sm S$.}

\begin{thm}\label{adhesions}
  Let $k\ge 2$, let $G$ be a graph, let $B_0,\ldots,B_n$ be a path decomposition of $G$, let $H_0:=G[B_0]$, and, for each $i\in\{1,\ldots,n\}$, let $H_i=G[B_i]-E(G[B_{i-1}\cap B_i])$.  If
  \begin{enumerate}[nosep,nolistsep,label=(\alph*)]
    \item\label{small_adhesions} $|B_{i-1}\cap B_i| \le k$ for each $i\in\{1,\ldots,n\}$, and
    \item\label{small_bases} $\bn(H_i)\le b$ for each $i\in\{0,\ldots,n\}$,
  \end{enumerate}
  then $\bn(G)\le b+ O(k\log^2 k)$.
\end{thm}

{\NoHyper\citeauthor{geniet.giocanti:basis}} prove their result on $K_t$-minor-free graphs using a version of \cref{adhesions} with a much larger dependence on $b$ and $k$. Using \cref{adhesions} and the recently established Polynomial Graph Minor Structure Theorem \cite{gorsky.sewery.ea:polynomial} in their argument shows that the function $f$ is polynomial: There exists a constant $c$ such that every $K_t$-minor-free $G$ has $\bn(G)\in O(t^c)$.

\section{The Proofs}

Throughout the remainder of this paper, we will sometimes treat sets of edges in a graph $G$ as subgraphs of $G$. So, for example, for a subgraph $H\subseteq G$ and a set $S\subseteq E(G)$, $H\cup S$ denotes the graph with vertex set $V(H\cup S):=V(H)\cup\bigcup_{vw\in S}\{v,w\}$ and edge set $E(H\cup S):=E(H)\cup S$.

A \defin{spanning forest} $F$ of a graph $G$ is a forest with the property that two vertices $v,w\in V(G)$ are in the same component of $G$ if and only if they are in the same component of $F$.  

\begin{lem}\label{double_forest}
  Let $G_1$ and $G_2$ be two graphs and, for each $i\in \{1,2\}$, let $F_i$ be a spanning forest of $G_i$.  For each pair of vertices $v,w\in V(G_1\cup G_2)$, $v$ and $w$ are in the same component of $G:=G_1\cup G_2$ if and only if $v$ and $w$ are in the same component of $F:=F_1\cup F_2$.
\end{lem}

\begin{proof}
  Let $v$ and $w$ be any two vertices of $G$.  If $v$ and $w$ are in different components of $G$ then they are in different components of $F$ since $E(F)\subseteq E(G)$.  If $v$ and $w$ are in the same component of $G$ then there is a path from $P=v_0,\ldots,v_\ell$ from $v$ to $w$ in $G$.  For each $j\in\{1,\ldots,\ell\}$, $v_{j-1}v_j$ is an edge of $G_i$ for at least one $i\in\{1,2\}$. Therefore, there exists a path $P_j$ in $F_i$ from $v_{j-1}$ to $v_{j}$ for at least one $i\in\{1,2\}$.  Therefore, the concatenation of the paths $P_1,\ldots,P_\ell$ is a walk in $F$ from $v$ to $w$, so $v$ and $w$ are in the same component of $F$.
\end{proof}

An edge $e$ in a graph $H$ is a \defin{cycle edge} if at least one cycle in $H$ contains $e$.  If $e$ is a cycle edge in a graph $H$ then, for any two vertices $v,w\in V(H)$, $v$ and $w$ are in the same component of $H$ if and only if they are in the same component of $H-e$.  Thus, one can obtain a spanning forest of $H$ by repeatedly removing an arbitrary cycle edge.

\begin{lem}\label{incremental}
  Let $G$ be a graph, let $v$ be a vertex of $G$, let $F^-$ be a spanning forest of $G-v$, let $\mathcal{B}^-$ be a cycle basis for $G-v$, and let $\Delta$ be a cycle basis for $F^+:=F^-\cup \{vw: w\in N_G(v)\}$.  Then $\mathcal{B}:=\mathcal{B}^-\cup\Delta$ is a generating set of $\CC(G)$.
\end{lem}

\noindent
Before proving \cref{incremental}, we note that the graph $F^+$ in the statement of \cref{incremental} consists of a forest $F^-$ (an outerplanar graph) plus one additional vertex $v$ adjacent to some vertices of $F^-$, so $F^+$ is a planar graph, which implies that $\bn(F^+)\le 2$.  When applying \cref{incremental}, this allows us to choose $\Delta$ so that $\ply(\Delta)\le 2$ and therefore $\ply(\mathcal{B})\le\ply(\mathcal{B}^-)+2$.

\begin{proof}[Proof of \cref{incremental}]
  Let $C$ be a cycle in $G$. We must show that $\mathcal{B}$ generates $C$.  If $C\subseteq G-v$ then $\mathcal{B}^-\subseteq\mathcal{B}$ generates $C$.  Otherwise, $C$ contains a three vertex path $P_{uvw}:=uvw$ with $u,w\in N_G(v)$.  Since $C-v\subseteq G-v$ is a path from $u$ to $w$, $u$ and $w$ are in the same component of $G-v$ and are therefore in the same component of $F^-$.  Let $P$ be the unique path from $u$ to $w$ in $F^-$ and let $C':=P\cup P_{uvw}$ be the unique cycle in $F^+$ that includes $P_{uvw}$.  By definition, $\Delta$ generates $C'$.  Finally, observe that $C\oplus C'\subseteq G-v$ is an Eulerian subgraph of $G-v$, so $\mathcal{B}^-$ generates $C\oplus C'$.  Therefore $\mathcal{B}=\mathcal{B}^-\cup\Delta$ generates $C'\oplus (C\oplus C')=C$.
\end{proof}
Let $F$ be a forest and let $S\subseteq V(F)$.  We denote by \defin{$\sop{F}{S}$}, the vertex-minimal subgraph of $F$ with the property that $v$ and $w$ are in the same component of $F$ if and only if $v$ and $w$ are in the same component of $\sop{F}{S}$, for each $v,w\in S$.  From this definition, it follows that every vertex in $V(\sop{F}{S})\sm S$ has degree at least $2$.  We denote by \defin{$\ssop{F}{S}$}, the labelled topological minor of $\sop{F}{S}$ obtained by suppressing every vertex $v\in V(\sop{F}{S})\sm S$ with $\deg_{\sop{F}{S}}(v)=2$.\footnote{To \defin{suppress} a degree-$2$ vertex $v$ with neighbours $u$ and $w$, we remove $v$ and replace it with the edge $uw$.}

\begin{lem}\label{edge_counter}
  For every forest $F$ and every $S\subseteq V(F)$ with $|S|\ge 2$, $\ssop{F}{S}$ has at most $2|S|-2$ vertices.
\end{lem}

\begin{proof}
  Let $T$ be a supergraph of $\ssop{F}{S}$ obtained by adding a minimal set of edges so that $T$ is connected.  Then $T$ is a tree.  Let $\ell$ be the number of leaves of $T$.  By the definition of $\ssop{F}{S}$, each vertex $v\in V(T)\sm S$ has $\deg_{T}(v)\ge \deg_{\ssop{F}{S}}(v)\ge 3$.  A well-known fact about trees is that the number of leaves in a tree with $t$ vertices of degree at least three is at least $t+2$.  Therefore, $|S|\ge \ell\ge |V(T)|-|S|+2$.  Rewriting this gives $|V(\ssop{F}{S})|=|V(T)|\le 2|S|-2$.  
\end{proof}

For a forest $F$ and two vertices $v$ and $w$ in the same component of $F$, let $P_F(v,w)$ denote the unique path in $F$ from $v$ to $w$.  For a graph $G$, generating set $\mathcal{B}$ of  $G$, a spanning forest $F$ of $G$ and a pair of vertices $v$ and $w$ in the same component of $G$ define $\ply_F(v,w,\mathcal{B})=\max\{\ply(e,\mathcal{B}):e\in E(P_F(v,w))\}$.  When $vw=e$ is an edge of some auxiliary graph (such as $\ssop{F}{S}$ for some $S\subseteq V(F)$), we use the shorthands $P_F(e):=P_F(v,w)$ and $\ply_F(e,\mathcal{B})=\ply_F(v,w,\mathcal{B})$.

\subsection{Proof of \cref{main}}
\label{proof_of_one}

Let $G$ be an $n$-vertex graph of pathwidth at most $t$. A path decomposition $B_0,\ldots,B_n$ of $G$ is \defin{normal} if
\[
   |B_i| = \begin{cases}
      i & \text{for $i\le t$} \\
      t+1 & \text{for $i\in\{t+1,\ldots,n\}$}
      \end{cases}
\]
and $|B_i\sm B_{i-1}|=1$, for each $i\in\{1,\ldots,n\}$.  It is well known that any graph of pathwidth at most $t$ has a normal path decomposition of width at most $t$.

\begin{proof}[Proof of \cref{main}]
  Graphs of pathwidth at most $1$ are forests and therefore have basis number zero. Thus $\bn(G)=0\le 4\pw(G)$ when $\pw(G)\le 1$.  
  Now, let $G$ be a graph of pathwidth at most $t\ge 2$ and let $B_0,\ldots,B_n$ be a normal path decomposition of $G$. We will show the existence of a generating set $\mathcal{B}$ for $G$ with $\ply(\mathcal{B})\le 4t$.
  We will prove the following technical statement, which is amenable to a proof by induction, and which immediately implies \cref{main}.

  \begin{clm}\label{technical}
    There exists a spanning forest $F$ of $G$ and a $4t$-generating set $\mathcal{B}$ of $\CC(G)$ such that, for each $c\in\{0,\ldots,2t-1\}$, the number of edges $e$ of $\ssop{F}{B_n}$ with $\ply_F(e,\mathcal{B})\ge 2c+1$ is at most $2t-c-1$.
  \end{clm}

  We now prove \cref{technical} by induction on $n$. The base cases $n\in\{0,1,2\}$ are trivial: In these cases $G$ has at most two vertices and no cycles, so taking $\mathcal{B}:=\emptyset$ and $F$ to be any spanning forest of $G$ satisfies the requirements, for $t\ge 1$.

  Now assume that $n\ge 3$.  Let $v$ be the unique vertex in $B_n\sm B_{n-1}$.  Then $B_0,\ldots,B_{n-1}$ is a normal path decomposition of $G-v$.  By the inductive hypothesis, $G-v$ has a spanning forest $F^-$ and a generating set $\mathcal{B}^-$ with $\ply(\mathcal{B}^-)\le 4t$ such that for each $c'\in\{0,\ldots,2t-1\}$ we have
\begin{equation}
\left|\{\, e\in \ssop{F^-}{B_{n-1}} :\ply_{F^-}(e,\mathcal{B}^-)\ge 2c'+1 \,\}\right| \leq 2t-c'-1.  \label{tech_induction}
\end{equation}

  Let $F_n:=G[B_n]-E(G[B_{n-1}\cap B_n])$ and observe that $V(F_n)= B_n$ and $E(F_n)=\{vw:w\in N_{B_{n-1}}(v)\}$.
  Since $F^+:=F^-\cup F_n$ is a planar graph, it has a cycle basis $\Delta$ with $\ply(\Delta)\le 2$. By \cref{incremental}, $\mathcal{B}:=\mathcal{B}^-\cup\Delta$ is a generating set for $\mathcal{C}(G)$. 

  What remains is to define the spanning forest $F$ of $G$ that satisfies the conditions of \cref{technical}.  
  By \Cref{double_forest}, two vertices are in the same component of $G$ if and only if they are in the same component of $F^+$.  To obtain $F$, we find a sequence of edges $\hat{e}_1,\ldots,\hat{e}_r$ such that $\hat{e}_i$ is a cycle edge in $F^+-\{\hat{e}_1,\ldots,\hat{e}_{i-1}\}$ that maximizes $\ply(\hat{e}_i,\mathcal{B})$. This process ends when $F^+-\{\hat{e}_1,\ldots,\hat{e}_r\}$ is a forest, at which point we take $F:=F^+-\{\hat{e}_1,\ldots,\hat{e}_r\}$.  Since this process only removes cycle edges, the resulting graph $F$ is a spanning forest of $F^+$ and therefore a spanning forest of $G$.

  We now show that $\mathcal{B}$ and $F$ satisfy the requirements of \cref{technical}.  
  We begin by showing that $\ply(\mathcal{B})\le 4t$.  
  Let $\hat{e}$ be an edge of $G$. If no cycle in $\Delta$ includes $\hat{e}$, then $\ply(\hat{e},\mathcal{B})=\ply(\hat{e},\mathcal{B}^-)\le\ply(\mathcal{B}^-)\le 4t$, by the inductive hypothesis.  
  Now suppose that at least one cycle $C$ in $\Delta$ includes $\hat{e}$.  Since $F^-=F^+-v$ is a forest, $C$ includes $v$.  
  If $\hat{e}$ is incident to $v$, then $\ply(\hat{e},\mathcal{B})=\ply(\hat{e},\Delta)\le 2\le 4t$.  Now suppose that $\hat{e}$ is an edge of $C-v$. 
  Then $C-v=P_{F^-}(u,w)$ for two vertices $u,w\in N_G(v)\subseteq B_{n-1}$.  Since $\{u,w\}\subseteq B_{n-1}$, $P_{F^-}(u,w)\subseteq \sop{F^-}{B_{n-1}}$.  Therefore, there exists an edge $e$ in $\ssop{F^-}{B_{n-1}}$ such that $\hat{e}\in E(P_{F^-}(e))$.  Observe that $P_{F^-}(e)\subseteq P_{F^-}(u,w)$.  \Cref{tech_induction} with $c'=2t-1$ states that $\ssop{F^-}{B_{n-1}}$ has at most $2t-c'-1=2t-(2t-1)-1=0$ edges $e'$ with $\ply_{F^-}(e',\mathcal{B}^-)\ge 2c'+1=4t-1$.  
 Therefore,  $\ply_{F^-}(e,\mathcal{B}^-)\le 4t-2$.
 By definition, we have  $\ply_F(v,w,\mathcal{B})=\max\{\ply(e):e\in E(P_F(v,w))\}$.
Since $\hat{e}$ is one of edges on that path, we infer that $\ply(\hat{e},\mathcal{B}^-)\le \ply_{F^-}(e,\mathcal{B}^-)$
 Therefore we have 
 \begin{align*}
    \ply(\hat{e},\mathcal{B})= &\ply(\hat{e},\mathcal{B}^-)+\ply(\hat{e},\Delta)\\
    \le&\ply_{F^-}(e,\mathcal{B}^-)+2\\
    \le & 4t-2+2=4t
 \end{align*}
Therefore, $\ply(e,\mathcal{B})\le 4t$ for each edge $e$ of $G$, so $\ply(\mathcal{B})\le 4t$.

  We now get to the crux of the proof; showing that $F$ satisfies the requirements of \cref{technical} for all $c\in\{0,\ldots,2t-1\}$.  For $c=0$, this follows immediately from \cref{edge_counter}. Indeed, $|B_n|\le t+1$ so, by \cref{edge_counter}, $\ssop{F}{B_n}$ has at most $2|B_n|-3 \le 2(t+1)-3=2t-1$ edges. 
  
  Next we consider the case where $c\in\{1,\ldots,2t-1\}$.  Let 
  \[ 
     S:=\{e\in E(\ssop{F}{B_n}):\ply_{F}(e,\mathcal{B})\ge 2c+1\}
  \]
  and let $m:=|S|$.  
  \[
    S^-:=\{e\in E(\ssop{F^-}{B_{n-1}}):\ply_{F^-}(e,\mathcal{B}^-)\ge 2(c-1)+1\}
  \]
  and let $m^-:=|S^-|$. By the inductive hypothesis, $m^- \le 2t-(c-1)-1$.  We will now prove that $m^- \ge m+1$, so $m+1\le m^- \le 2t-(c-1)-1$ and rewriting this inequality gives $m\le 2t-c-1$.  We begin by describing an injective function $\varphi:S\to S^-$. The existence of $\varphi$ immediately implies that $m=|S|\le|S^-|= m^-$.  To establish that $m\le m^--1$, we will then show that there is at least one element of $S^-$ with no preimage in $\varphi$.

  Let $e$ be an edge of $\ssop{F}{B_n}$ that is in $S$.
  Since $\ply_F(e,\mathcal B)=\max\{\ply(f,\mathcal B): f\in E(P_F(e))\}$,
there exists an edge $\hat e\in E(P_F(e))$ with
$\ply(\hat e,\mathcal B)=\ply_F(e,\mathcal B)\ge 2c+1$.
  Since $2c+1\ge 3$, the edge $\hat{e}$ is not incident to $v$. Therefore $\hat{e}$ is an edge of $F^-$.  
  We note that $\ply(\hat{e},\mathcal{B})=\ply(\hat{e},\mathcal{B}^-)+\ply(\hat{e},\Delta)$.
  Since at most two cycles in $\Delta$ use the edge  $\hat{e}$, $\ply(\hat{e},\mathcal{B}^-)\ge\ply(\hat{e},\mathcal{B})-2\ge 2(c-1)+1$.  

  Since $e$ is an edge of $\ssop{F}{B_n}$, $P_F(e)$ is a path in $\sop{F}{B_n}$.  By the minimality of $\sop{F}{B_n}$, there are two vertices $u,w\in B_n$ such that the path from $u$ to $w$ in $\sop{F}{B_n}$ includes $P_F(e)$.  Since $\hat{e}$ is not incident to $v$ and $N_G(v)\subseteq B_{n}\sm\{v\}$, we may assume that $v\not\in\{u,w\}$.\footnote{If, for example $v=u$ then $P_F(v,w)-v$ is a path from $u'$ to $w$ in $F$ with $u',w\in B_{n}\sm\{v\}$ that includes the edge $\hat{e}$.} Therefore $P_F(u,w)$ is a path in $F$ that avoids $v$, which implies that $P_F(u,w)$ is a path in $F^-$. Therefore $P_F(u,w)=P_{F^-}(u,w)$.  Since $B_n\sm B_{n-1}=\{v\}$, $u,w\in B_{n-1}$. Therefore $P_{F^-}(u,w)$ is a path in $\sop{F^-}{B_{n-1}}$.  Therefore $\ssop{F^-}{B_{n-1}}$ contains an edge $e'$ such that $P_{F^-}(e')$ contains the edge $\hat{e}$, so $\ply_{F^-}(e',\mathcal{B}^-)\ge\ply(\hat{e},\mathcal{B}^-)\ge 2(c-1)+1$.  Therefore $e'\in S^-$ and we set $\mathdefin{\varphi(e)}:=e'$.  Observe that $P_{F}(e)=P_{F^-}(e)$ contains $P_{F}(e')=P_{F^-}(e')$ as a subpath.  The fact that 
  $P_{F^-}(e')\subseteq P_{F^-}(e)$ implies that $\varphi$ is an injective function from $S$ to $S^-$. The existence of $\varphi$ already implies that $m^-\ge m$.

  To show that $m^-\ge m+1$, we will show that there is an edge $e_1\in S^-$ that has no preimage in $\varphi$. By the inductive hypothesis, the number $m^\star$ of edges $e$ of $\ssop{F^-}{B_{n-1}}$ with $\ply(e,\mathcal{B}^-)\ge 2c+1$ satisfies $m^\star\le 2t-c-1$.  If $m\le m^\star$ then there is nothing more to prove, so we may assume that $m\ge m^\star+1$.  Therefore, there exists an edge $e$ of $\ssop{F}{B_n}$ with $\ply_{F^-}(e,\mathcal{B})\ge 2c+1$ and $\ply_{F^-}(e,\mathcal{B}^-)< 2c+1$.  Therefore $P_{F^-}(e)$ includes an edge $\hat{e}$ with $\ply(\hat{e},\mathcal{B})\ge 2c+1$ and $\ply(\hat{e},\mathcal{B}^-)< 2c+1$.
  Therefore, at least one cycle in $\Delta$ includes $\hat{e}$. Therefore $\Delta$ includes at least one cycle and so $F^+$ includes at least one cycle.  Therefore, there is at least one edge $\hat{e}_1$ in $F^+$ that is not in $F$. The edge $\hat{e}_1$ was chosen to remove from $F^+$ because $\hat{e}_1$ is a cycle edge of $F^+$ that maximizes $\ply(\hat{e}_1,\mathcal{B})$.  Since $e$ is a cycle edge of $F^+$, $\ply(\hat{e}_1,\mathcal{B})\ge\ply(\hat{e},\mathcal{B})\ge 2c+1$. 
  
  Therefore, $\ply(\hat{e}_1,\mathcal{B}^-)\ge \ply(\hat{e}_1,\mathcal{B})-2\ge 2(c-1)+1$.  Then $C-v$ is a path in $F^-$ with endpoints $u,w\in B_{n}\sm\{v\}\subseteq B_{n-1}$.  Therefore $C-v=P_{F^-}(u,w)$ is a path in $\sop{F^-}{B_{n-1}}$. Therefore $\ssop{F^-}{B_{n-1}}$ contain an edge $e_1$ such that $P_{F^-}(e_1)$ contains the edge $\hat{e}_1$. Therefore $\ply_{F^-}(e_1,\mathcal{B}^-)\ge \ply(\hat{e}_1,\mathcal{B}^-)\ge 2(c-1)+1$.  However, since $\hat{e}_1$ is not in $F$, there is no edge $e^\star$ of $\ssop{F}{B_n}$ such that $P_F(e^\star)$ contains $\hat{e}_1$. Therefore, there is no edge $e^\star$ of $\ssop{F}{B_n}$ such that $\varphi(e^\star)=e_1$.  Therefore, the image of $\varphi$ contains at least $m$ edges in $S^-$ and there is at least one edge $e^\star\in S^-$ that is not in the image of $\varphi$. Therefore $m^-=|S^-|\ge|S|+1= m+1$, which completes the proof.
\end{proof}

\subsection{Proof of \cref{adhesions}}
\label{proof_of_two}

Our proof of \cref{adhesions} makes use of the following generalization of \cref{incremental}:\footnote{\cref{incremental} is the special case in which $H_1:=G-v$, $F:=F_1$ and $H_2:=F_2:=\{vw: w\in N_G(v)\}$.}

\begin{lem}\label{merging}
  Let $H_1$ and $H_2$ be two graphs; for each $i\in\{1,2\}$, let $F_i$ be spanning forest of $H_i$ and let $\mathcal{B}_i$ be a cycle basis for $H_i$; and let $\Delta$ be a cycle basis for $F_1\cup F_2$.  Then $\mathcal{B}:=\mathcal{B}_1\cup \mathcal{B}_2\cup \Delta$ is a generating set  for $\CC(H)$, where $H:=H_1\cup H_2$.
\end{lem}

\begin{proof}
  Let $A:=V(H_1)\cap V(H_2)$. 
  Let $C$ be a cycle in $H$.  We must show that $\mathcal{B}$ generates $C$.  An \defin{$A$-segment} of $C$ is a path in $C$ with both endpoints in $A$ and no internal vertex in $A$.  Note that each $A$-segment $P$ of $C$ has $E(P)\subseteq E(H_1)$ or $E(P)\subseteq E(H_2)$.  An $A$-segment $P$ of $C$ is \defin{clean} if $E(P)\subseteq E(F_1\cup F_2)$ and is \defin{dirty} otherwise.  The proof is by induction on the number, $d$, of dirty $A$-segments of $C$.  If $d=0$, then $E(C)\subseteq E(F_1\cup F_2)$, so $\Delta\subseteq\mathcal{B}$ generates $C$.

  Now assume that $d\ge 1$ and let $P$ be a dirty $A$-segment of $C$ with endpoints $v,w\in A$.  Without loss of generality, we may assume that $E(P)\subseteq E(H_1)$.  Since $P\subseteq H_1$, $v$ and $w$ are in the same component of $H_1$.  Therefore $v$ and $w$ are in the same component of $F_1$, so the path $P_1:=P_{F_1}(v,w)$ exists.  
  Let $C':= P_1\oplus C\oplus P$ and we note that $C'$ is Eulerian.
  Let $C'_1$ be a cycle in $C'$ that contains $E(P_1)$. Let $C_2',\ldots,C'_k$ be a partition of the edges of $C'\oplus C'_1=C'- E(C'_1)$ into pairwise edge-disjoint cycles.  (Such cycles $C'_2,\ldots,C'_k$ exist, by Veblen's Theorem). 
  Except for $A$-segments contained in $P_1$, each $A$-segment of $C_i'$ is also an $A$-segment of $C$, for each $i\in\{1,\dots,k\}$.  
  Since $P$ is dirty and each $A$-segment contained in $P_1$ is clean, the number $d_i$ of dirty $A$-segments in $C'_i$ is less than $d$, for each $i\in\{1,\ldots,k\}$. (In fact $\sum_{i=1}^k d_i < d$.) By the inductive hypothesis, $\mathcal{B}$ generates $C'_i$, for each $i\in\{1,\ldots,k\}$.  
  Therefore $\mathcal{B}$ generates $C'=\bigoplus_{i=1}^k C'_i$.  Finally, notice that $P_1\oplus P$ is an Eulerian subgraph of $H_1$, so $\mathcal{B}_1$ generates $P_1\oplus P$.  Therefore $\mathcal{B}$ generates $C' \oplus P_1 \oplus P
  = \bigl(P_1 \oplus C \oplus P\bigr) \oplus P_1 \oplus P
  = C$.
\end{proof}
 
We also make use of the following result, which follows quickly from \cref{general_upper_bound}:

\begin{lem}\label{two_trees}
  Let $T_1$ and $T_2$ be two trees with $|V(T_1)\cap V(T_2)|\le k$.  Then $\bn(T_1\cup T_2) \in O(\log^2 k)$.
\end{lem}

\begin{proof}
  Let $A:=V(T_1)\cap V(T_2)$.  Observe that any cycle $C$ in $T_1\cup T_2$ is a cycle in $H:=\sop{T_1}{A}\cup\sop{T_2}{A}$, so it suffices to show that $\bn(H)\in O(\log^2 k)$.  Consider the graph $\tilde{H}:=\ssop{T_1}{A}\cup\ssop{T_2}{A}$.
  It follows from \cite[Lemma 3.6]{biedl} that  $\bn(H)=\bn(\tilde{H})$.  By \cref{edge_counter}, $\tilde{H}$ has at most $4k - 4$ vertices. By \cref{general_upper_bound}, $\bn(\tilde{H})\in O(\log^2 k)$.  Therefore $\bn(H)=\bn(\tilde{H})\in O(\log^2 k)$.
\end{proof}

\begin{proof}[Proof of \cref{adhesions}]
  Let $G$ be a graph and let $B_0,\ldots,B_n$ be a path decomposition of $G$ satisfying conditions \ref{small_adhesions} and \ref{small_bases} of \cref{adhesions}.  Our proof is by induction on $n$, but with a stronger inductive hypothesis that makes use of the following definitions. For each $i\in\{-1,0,\ldots,n\}$, define $G_i:=G[\bigcup_{j=0}^i B_j]$.   For each $i\in\{0,\ldots,n-1\}$, let $A_i:=B_{i}\cap B_{i+1}$.
  For any spanning forest $F$ of $G$ and any $i\in\{0,\ldots,n\}$, we say that an edge $e$ of $\ssop{F}{A_i}$ is \defin{$i$-old} if $P_F(e)\subseteq G_i$.

  We first deal with the cases where $k\le 1$. In such cases $A_i$ is a cutset of size at most $1$ that separates $\bigcup_{j=0}^{i} B_j$ from $\bigcup_{j=i+1}^n B_j$, for each $i\in\{0,\ldots,n-1\}$. Then any cycle of $G$ is contained in $G_i$ for some $i\in\{0,\ldots,n\}$.  For each $i\in\{0,\ldots,n\}$, let $\Lambda_i$ be a cycle basis of $G_i$ having $\ply(\Lambda_i)\le b$.  Then $\mathcal{B}:=\bigcup_{i=0}^n\Lambda_i$ is a generating set for $\CC(G)$ with $\ply(\mathcal{B})\le b$, so $\bn(G)\le b\le b + O(k\log^2 k)$.  We now assume that $k\ge 2$.  Since any graph that satisfies the conditions of \cref{adhesions} for $k=2$ also satisfies these conditions for $k=3$, we now assume that $k\ge 3$.

  \begin{clm}\label{i_old}
    There exist a constant $c>0$ such that for any $A_n\subseteq B_n$ of size at most $k$, there exists a generating set  $\mathcal{B}$ of $\mathcal C(G)$ with $\ply(\mathcal{B})\le b+ (2k-2)c\log^2 k$ and a spanning forest $F$ of $G$ such that, for each $i\in\{0,\ldots,n\}$, $\ssop{F}{A_i}$ has at most
    \begin{equation}
       2k-2-\frac{\max\{\ply(\hat{e},\mathcal{B}):\hat{e}\in E(G_i)\sm E(G_{i-1})\} - b}{c\log^2 k} \label{yucko}
    \end{equation}
    edges that are $i$-old.
  \end{clm}

  The remainder of this proof is dedicated to proving \cref{i_old}, from which \cref{adhesions} immediately follows.  The proof is by induction on $n$.  
  Let 
  \[ \mathdefin{z_i}:=\frac{\max\{\ply(\hat{e},\mathcal{B}):\hat{e}\in E(G_i)\sm E(G_{i-1})\} - b}{c\log^2 k} \enspace ,
  \]
  so that the second part of \cref{i_old} is equivalent to stating that at most
  \[
    2k-2-z_i
  \]
  edges of $\ssop{F}{A_i}$ are $i$-old, for each $i\in\{0,\ldots,n\}$.  The guiding intuition behind the rest of this proof is that increasing $z_i$ by one results in the elimination of at least one $i$-old edge.

  The base case $n=0$ is straightforward:  By assumption, $G_0=G[B_0]$ has basis number $\bn(G_0)\le b$.  Let $\mathcal{B}$ be any basis of $G_0$ with ply at most $b$ and let $F$ be any spanning forest of $G_0$.  Then, for each edge $\hat{e}$ of $G_0$, $\ply_F(\hat{e},\mathcal{B})\le b$, so $z_0=0$.  By \cref{edge_counter}, $\ssop{F}{A_0}$ has at most $2k-2$ vertices and at most $2k-3$ edges (and therefore at most $2k-3< 2k-2-z_0$ edges that are $0$-old), as required by \cref{i_old}.

  We now assume that $n\ge 1$.
  Apply \cref{i_old} inductively with $A_{n-1}:=B_{n-1}\cap B_n$ to the graph $G_{n-1}$ with the path decomposition $B_0,\ldots,B_{n-1}$ to obtain a cycle basis $\mathcal{B}^-$ and a spanning forest $F^-$ of $G_{n-1}$.  Let $\Lambda$ be a cycle basis for $G[B_n]-E(G[A_{n-1}])$ with $\ply(\Lambda)\le b$, and let $F_n$ be any spanning forest of $G[B_n]$.  By \cref{two_trees}, $F^+:=F^-\cup F_n$ has a cycle basis  $\Delta$ with $\ply(\Delta)\le c \log^2 k$.  We define $\mathcal{B}:=\mathcal{B}^-\cup \Lambda \cup \Delta$.  By \cref{merging}, $\mathcal{B}$ is a generating set for $\CC(G_n)=\CC(G)$.

  All that remains is to define a forest $F$ that satisfies the conditions of \cref{i_old}.  To obtain $F$, we find a sequence of edges $\hat{e}_1,\ldots,\hat{e}_r$ such that $\hat{e}_i$ is a cycle edge in $F^+-\{\hat{e}_1,\ldots,\hat{e}_{i-1}\}$ whose birth time is minimum. This process ends when $F^+-\{\hat{e}_1,\ldots,\hat{e}_r\}$ is a forest, at which point we take $F:=F^+-\{\hat{e}_1,\ldots,\hat{e}_r\}$.  All that remains is to show that $\mathcal{B}$ and $F$ satisfy the requirements of \cref{i_old}.

  We first argue that $\ply(\mathcal{B})\le b+ (2k-2)c\log^2 k$, assuming that the forest $F$ satisfies the conditions of \cref{i_old}. For each edge $\hat{e}$ of $G$, define the \defin{birth time} of $\hat{e}$ as $\min\{i:\hat{e}\in E(G_i)\}$.  Note that, for each $i\in\{0,\ldots,n\}$, the edges of birth time $i$ are precisely the edges in $E(G_i)\sm E(G_{i-1})$ that determine the value of $z_i$.  Let $\hat{e}$ be an edge of $G$ that maximizes $\ply(\hat{e},\mathcal{B})$ and let $i$ be the birth time of $\hat{e}$.  Suppose, by way of contradiction, that $\ply(\hat{e},\mathcal{B})> b+(2k-2)\cdot c\log^2 k$.  
  Since $\hat{e}$ has birth time $i$, $z_i> 2k-2$.  Then $\ssop{F}{A_i}$ has most $2k-2 - z_i < 0$ edges that are $i$-old.  This is clearly a  contradiction, since the number of $i$-old edges of $\ssop{F}{A_i}$ cannot be less than zero.

  We now show that $\mathcal{B}$ and $F$ satisfy the remaining condition of \cref{i_old}.
  Fix $i\in\{0,\ldots,n\}$ and let $\hat{e}$ be an edge of $G$ with birth time $i$ that maximizes $\ply(e,\mathcal{B})$, so $\hat{e}$ determines the value of $z_i$.  Let $m_i$ be the number of edges of $\ssop{F}{A_i}$ that are $i$-old.  We must show that $m_i\le 2k-2-z_i$.

  We first consider the case where $i=n$. In this case, $\hat{e}\not\in E(G_{n-1})$, so $\ply(\hat{e},\mathcal{B}) = \ply(\hat{e},\Lambda)+\ply(\hat{e},\Delta) \le b + c\log^2 k$. In this case, $z_i\le 1$. By \cref{edge_counter}, $\ssop{F}{A_n}$ has at most $2k-3\le 2k-2-z_i$ edges, so $m\le 2k-2-z_i$, as required.

  Next we consider the case where $i\in\{0,\ldots,n-1\}$. In this case, $\ssop{F^-}{A_i}$ contains an $i$-old edge $e$ such that $P_{F^-}(e)$ contains the edge $\hat{e}$.  Since $i < n$, $\hat{e}\in E(G_{n-1})$. Therefore, no cycle in $\Lambda$ contains $\hat{e}$, so $\ply(\hat{e},\mathcal{B})=\ply(\hat{e},\mathcal{B}^-)+\ply(\hat{e},\Delta)\le \ply(\hat{e},\mathcal{B}^-)+c\log^2 k$.  Let $S_i$ be the set of $i$-old edges of $\ssop{F}{A_i}$ and let $S_i^-$ be the set of $i$-old edges of $\ssop{F^-}{A_i}$.  Let $m_i=|S_i|$ and $m_i^-:=|S_i^-|$.
  For each edge $f\in S_i$, there is at least one edge $f^-\in S_i^-$ such that  $P_{F^-}(f^-)\subseteq P_F(f)$. Thus, there is an injective mapping $\varphi:S\to S^-$.  This immediately implies that $m_i\le m_i^-$.  Let $z_i^-$ be such that $\ply(\hat{e},\mathcal{B}^-)= b+z_i^-\cdot c\log^2 k$.  The inductive hypothesis implies that $m_i^- \le 2k-2-z_i^-$.  Since $\mathcal{B}\supseteq\mathcal{B}^-$, $\ply(\hat{e},\mathcal{B})\ge \ply(\hat{e},\mathcal{B}^-)$ and $z_i\ge z_i^-$.  If $z_i=z_i^-$ then $m_i^-\le 2k-2-z_i^-=2k-2-z_i$, so $m_i\le m_i^-\le 2k-2-z_i$, as required.

  To finish, we consider the case where $z_i > z_i^-$ and show that this implies that $S^-$ contains an edge $e^-_1$ with no preimage in $\varphi$.  
  Since $z_i>z_i^-$, $\ply(\hat{e},\mathcal{B})>\ply(\hat{e},\mathcal{B}^-)$. Since $\ply(\hat{e},\mathcal{B})=\ply(\hat{e},\mathcal{B}^-)+\ply(\hat{e},\Delta)$, some cycle in $\Delta$ contains $\hat{e}$. Therefore, $F^+$ contains a cycle that contains $\hat{e}$.  Therefore $F^+$ contains a cycle that contains an edge whose birth time is at most $i$.  Since $\hat{e}_1$ has minimum birth time over all edges of $F^+$ that participate in cycles, $\hat{e}_1$ has birth time at most $i$.  Then $\ssop{F^-}{A_i}$ contains an edge $e_1^-$ such that $\hat{e}_1$ is contained in $P_{F^-}(e_1^-)$.  Therefore $e_1^-$ is an $i$-old edge of $\ssop{F^-}{A_i}$, so $e_1^-\in S^-$.
  Since $\hat{e}_1$ is not in $F$, there is no edge $e_1$ in $\ssop{F}{A_i}$ such that $P_F(e_1)$ contains $\hat{e}_1$. Therefore $e_1^-\in S^-$ has no preimage in $\varphi$.  Therefore $m_i=|S|<|S^-| = m^-_i$.  Since $m_i$ and $m^-_i$ are integers, $m_i \le m^-_i -1\le 2k-2-z_i^--1\le 2k-2-z_i$. This completes the proof.
\end{proof}

\section{Conclusions}

The proofs of \cref{main} and \cref{adhesions} are quite similar, the main difference being how the edges $\hat{e}_1,\ldots,\hat{e}_r$ are chosen.
Indeed, \cref{main} can also be established using the second argument: a normal path decomposition $(B_0,\ldots,B_{n})$ of a pathwidth-$t$ graph has adhesions of size at most $t$; the cycle basis $\Lambda$ can be empty, since $G_n-E(G_{n-1})$ is a star, and the cycle basis $\Delta$ can be chosen so that $\ply(\Delta)\le 2$ (rather than $c\log^2 k$), since $F^-\cup G_i-E(G_{i-1})$ is planar.

Our reasons for presenting both proofs is that both ideas were tried (unsuccessfully, so far) in trying to prove the following result:

\begin{conj}
  Every graph $G$ of treewidth at most $k$ has $\bn(G)\in O(k)$.
\end{conj}

A second reason for presenting both proofs is that the elegant proof of \cref{general_upper_bound} by \citet{freedman.hastings:building} is a probabilistic argument that repeatedly adds a (shortest) cycle $C$ to add to the generating set and then chooses a random edge of $C$ to remove from $G$.  In an experimental work, \citet{wang.irani:cycle} show that choosing an edge that (deterministically or probabilistically) maximizes ply (as in the proof of \cref{main}) produces better cycle bases in practice. The hope is that insights gained from the proof of \cref{main} could be helpful in establishing the following conjecture:

\begin{conj}{\rm \cite{biedl}}
  Every $n$-vertex graph $G$ has $\bn(G)\in O(\log n)$.
\end{conj}

\bibliographystyle{plainurlnat}
\bibliography{bnpw}
% \nocite{*}

\end{document}